\documentclass{amsart}
\usepackage{amssymb}
\usepackage{mathrsfs}
\usepackage{cases}
\usepackage{amssymb,amscd,amsthm}
\usepackage{latexsym}
\date{\today}

\newcommand{\Z}{{\mathbb Z}}
\newcommand{\R}{{\mathbb R}}

\newtheorem{theorem}{Theorem} [section]

\newtheorem{lemma}[theorem]{Lemma}
\newtheorem{prop}[theorem]{Proposition}

\newtheorem{definition}[theorem]{Definition}

\newcommand{\p}{\|}

\newcommand{\be}{\begin{equation}}
\newcommand{\ee}{\end{equation}}

\DeclareMathOperator{\dist}{dist}

\newcommand{\eps}{\varepsilon}

\newcommand{\sig}{\sigma}

\newcommand{\om}{\omega}
\begin{document}

\title{Limit-Periodic Schr\"odinger Operators on $\Z^d$: Uniform Localization}

\author[D.\ Damanik]{David Damanik}

\address{Department of Mathematics, Rice University, Houston, TX~77005, USA}



\email{damanik@rice.edu}

\urladdr{www.ruf.rice.edu/$\sim$dtd3}

\author[Z.\ Gan]{Zheng Gan}

\address{Department of Mathematics, Rice University, Houston, TX~77005, USA}



\email{zheng.gan@rice.edu}

\urladdr{math.rice.edu/$\sim$zg2}

\thanks{D.\ D.\ and Z.\ G.\ were supported in part by NSF grant
DMS--0800100.}


\begin{abstract}
We exhibit $d$-dimensional limit-periodic Schr\"odinger operators that are uniformly localized in the strongest sense possible. That is, for each of these operators, there is a uniform exponential decay rate such that every element of the hull of the corresponding Schr\"odinger operator has a complete set of eigenvectors that decay exponentially off their centers of localization at least as fast as prescribed by the uniform decay rate. Consequently, these operators exhibit uniform dynamical localization.
\end{abstract}

\maketitle

\section{Introduction}
Localization is a topic that has been explored in the context of Schr\"odinger operators to a great extent. In \cite{dg3}, we investigated uniform localization in one-dimensional limit-periodic Schr\"odinger operators. We refer the reader to that paper for a discussion of history and context. When presenting the results from \cite{dg3}, we were asked by a number of people whether there are analogous results in higher dimensions. The purpose of this paper is to give an affirmative answer. That is, we generalize the results from \cite{dg3} and exhibit the phenomenon of uniform localization in multi-dimensional limit-periodic Schr\"odinger operators.

Before introducing the relevant Schr\"odinger operators, we give some preliminary definitions first that will be necessary to describe the resulting hulls.

\begin{definition}\label{d.cgmt}
{\rm (a)} We say that $\Omega$ is a \emph{Cantor group} if it is an infinite, totally disconnected, metrizable, compact Abelian group. We fix a metric $\mathrm{dist}$ on $\Omega$ that is compatible with the topology.

{\rm (b)} Consider a Cantor group $\Omega$ and a $\Z^d$ action by translations, $\{ T^n \}_{n \in \Z^d}$. That is, there are $\alpha_1 , \ldots , \alpha_d \in \Omega$ such that for $\omega \in \Omega$, we have
\begin{equation}\label{e.taction}
T^n \omega = \omega + \sum_{j = 1}^d n_j \alpha_j,
\end{equation}
where we write the group operation as $+$.\footnote{While $+$ is a natural way to denote the group operation in the abstract setting, for the concrete groups that arise as hulls of limit-periodic elements of $\ell^\infty(\Z^d)$, this is ambiguous. Thus, in the concrete setting, we will prefer to use $\cdot$ to denote the group operation.} We say that the action is \emph{minimal} if all orbits are dense, that is, for each $\omega \in \Omega$, we have $\overline{\{ T^n \omega : n \in \Z^d \}} = \Omega$.
\end{definition}

Given a Cantor group $\Omega$ that admits a minimal $\Z^d$ action $T$ by translations, we consider Schr\"odinger operators $H_\omega$ acting on $\ell^2(\Z^d)$ as
\begin{equation}\label{equ:oper}
(H_\omega u)(n) = \sum_{|l|_1=1}u(n+l) + V_\omega(n) u(n),
\end{equation}
where
\begin{equation}\label{equ:pot}
V_\omega (n) = f(T^n  \omega ) , \quad \omega \in \Omega, \; n \in
\Z^d
\end{equation}
with $f \in C(\Omega, \R)$.

\begin{definition}
We say that a family $\{u_n \} \subset \ell^2(\Z^d)$ is uniformly localized if there exist constants $r > 0$, called the decay rate, and $C < \infty$
such that for every element $u_n$ of the family, one can find $m_n \in \Z^d$, called the center of localization, so that $|u_n(k)| \leq C e^{-r|k-m_n|}$ for every $k \in \Z^d$. We say that the operator $H_\omega$ has ULE if it has a complete set of uniformly localized eigenfunctions.\footnote{Recall that a set of vectors is called complete if their span (i.e., the set of finite linear combinations of vectors from this set) is dense.}
\end{definition}

The notion of uniformly localized eigenfunctions and related ones were introduced by del Rio et al.\ in their comprehensive study of the question ``What is localization?'' \cite{dj2,dj}. As explained there, ULE implies uniform dynamical localization, that is, if $H_\omega$ has ULE, then
\begin{equation}\label{e.udl}
\sup_{t \in \R} \left| \left\langle \delta_n , e^{-itH_\omega} \delta_m \right\rangle \right| \le C_\omega e^{-r_\omega |n-m|}
\end{equation}
with suitable constants $C_\omega , r_\omega \in (0,\infty)$. While both properties are desirable, they are extremely rare.

The occurrence of pure point spectrum for the operators $\{H_\omega\}_{\omega \in \Omega}$ is called \textit{phase stable} if it holds for every $\omega \in \Omega$. It is an unusual phenomenon since most known models are not phase stable. It is known that uniform localization of eigenfunctions (ULE) has a close connection with phase stability of pure point spectrum. Indeed, it was shown in \cite[Theorem C.1]{dj} that if a $\mu$-ergodic family $\{H_\omega\}_{\omega \in \Omega}$ of Schr\"odinger operators on $\ell^2(\Z^d)$ with continuous sampling function has ULE for $\om$ in a set of positive $\mu$-measure, then $H_\om$ has pure point spectrum for every $\om \in \mathrm{supp}(\mu)$, where $\mathrm{supp}(\mu)$ denotes the topological support of $\mu$, that is, the complement of the largest open set $S \subset \Omega$ for which $\mu(S) = 0$. Jitomirskaya pointed out in \cite{j} that this result can be strengthened for a minimal $T$ in the sense that if there exists some $\om_0$ such that $H_{\om_0}$ has ULE, then $H_\om$ has pure point spectrum for every $\om \in \mathrm{supp}(\mu)$.

Our main result is the following.

\begin{theorem} \label{thm:main}
There exists a Cantor group $\Omega$ that admits a minimal $\Z^d$ action $T$ by translations, and an $f \in C(\Omega,\R)$ such that for every $\om \in \Omega$ the Schr\"odinger operator with potential $f(T^n \om)$ has ULE with $\omega$-independent constants. In particular, we have uniform dynamical localization \eqref{e.udl} for every $\omega$  with $\omega$-independent constants as well.
\end{theorem}

The proof is constructive and produces a family of pairs $(\Omega,f)$ for which the statement holds. The precise description of this family requires the material we will present in Section~\ref{sec:Cantor} below.

As in \cite{dg3}, we will still heavily use P\"oschel's results from \cite{p} (which will be recalled in Section~\ref{sec:poeschel}) to obtain the above theorem. The key step in proving Theorem~\ref{thm:main} is to construct a distal limit-periodic potential in our framework, which will be done in Section~\ref{sec:distal}. The proof of Theorem~\ref{thm:main} will be presented in Section~\ref{sec:localization}.

\section{The Connection Between Limit-Periodicity and Cantor Groups} \label{sec:Cantor}

In this section we establish a connection between the hulls of limit-periodic elements of $\ell^\infty(\Z^d)$ and Cantor groups which admit a minimal $\Z^d$ action by translations. For $d=1$, this was worked out in detail in \cite[Section~2]{a}; here we present its generalization to the general case. While this connection is probably well known to experts, its presentation is still worthwhile since it is fundamental for this paper and all its predecessors \cite{a, dg1, dg2, dg3}.

\begin{definition}\label{d.hull}
Let $d \in \Z_+$. The group $\Z^d$ acts on $\ell^\infty(\Z^d)$ as $(S_m V)(n) = V(n - m)$ for $n,m \in \Z^d$ and $V \in \ell^\infty(\Z^d)$. The set $\mathrm{orb}(V) = \{ S_m V : m \in \Z^d \}$ is called the \emph{orbit} of $V$ and the closure of its orbit is called its \emph{hull}, that is, $\mathrm{hull}(V) = \overline{\mathrm{orb}(V)}$. An element $V$ of $\ell^\infty(\Z^d)$ is called \emph{periodic} if its orbit is finite. It is called \emph{limit-periodic} if it belongs to the closure of the set of periodic elements of $\ell^\infty(\Z^d)$.
\end{definition}

$V$ is periodic in the sense of Definition~\ref{d.hull} if and only if it is periodic in each direction, that is, there are $p_1 , \ldots , p_d \in \Z_+$ such that for all $n = (n_1, \ldots , n_d)$, $k = (k_1 , \ldots , k_d) \in \Z^d$, we have $V(n_1 + k_1 p_1 , \ldots, n_d + k_d p_d) = V(n_1 , \ldots , n_d)$. We will call $p = (p_1,\ldots,p_d) \in (\Z_+)^d$ a \textit{periodicity vector} of $V$.

The following proposition describes how limit-periodic elements of $\ell^\infty(\Z^d)$ may be generated.

\begin{prop}\label{p.2}
Suppose $\Omega$ is a Cantor group that admits a minimal $\Z^d$ action by translations, $\{ T^n \}_{n \in \Z^d}$. Then, for every $f \in C(\Omega,\R)$ and every $\omega \in \Omega$, the element $V_\omega$ of $\ell^\infty(\Z^d)$ defined by $V_\omega(n) = f(T^n \omega)$ is limit-periodic. Moreover, for each $\omega \in \Omega$, we have $\mathrm{hull}(V_\omega) = \{ V_{\tilde \omega} : \tilde \omega \in \Omega \}$.
\end{prop}

We first prove the following simple lemma:

\begin{lemma}\label{l.stronglyminimal}
Suppose that $\{ T^n \}_{n \in \Z^d}$ is an action by translations as in \eqref{e.taction} on the compact Abelian group $\Omega$. Then, for each $j \in \{1 , \ldots , d\}$, there is a sequence $\{ n^{(j)}_k\}_{k \in \Z_+} \subset \Z_+$ such that $\lim_{k \to \infty} n^{(j)}_k \alpha_j = \omega_e$, the identity element of $\Omega$.
\end{lemma}

\begin{proof}
Let us fix $j$ and explain how to find $\{ n^{(j)}_k\}_{k \in \Z_+} \subset \Z_+$. Since $\Omega$ is compact, there exists an increasing sequence of positive integers $m_k \to \infty$ such that $m_k \alpha_j$ converges to some $\omega \in \Omega$ as $k \to \infty$. For each $k$, choose $\tilde m_k \in \{ m_{k + \ell} : \ell \ge 1 \}$ such that $n^{(j)}_k : =  \tilde m_k - m_k \ge k$. Then, $n^{(j)}_k \to \infty$ as $k \to \infty$ and
$\lim_{k \to \infty} n^{(j)}_k \alpha_j = \omega - \omega = \omega_e$, as desired.
\end{proof}

\begin{proof}[Proof of Proposition~\ref{p.2}.]
For a given $\varepsilon > 0$, we may choose a compact open neighborhood $U$ of the identity $\omega_e \in \Omega$ that is small enough so that $| f(\omega + \omega_U) - f(\omega) | < \varepsilon$ for every $\omega_U \in U$ and every $\omega \in \Omega$.

Since $U$ is compact and open, we can choose $\delta > 0$ such that $\dist (\omega_U , \omega_{\Omega \setminus U}) > \delta$ for every $\omega_U \in U$ and every $\omega_{\Omega \setminus U} \in \Omega \setminus U$.

Lemma~\ref{l.stronglyminimal} shows that we can choose $p_1,\ldots,p_d \in \Z_+$ such that $d(\omega_e,p_j \alpha_j) < \delta$ for $j = 1 ,\ldots, d$. By the defining property of $\delta$, it follows that the closure of
$$
\left\{ \sum_{j=1}^d n_j p_j \alpha_j : n = (n_1,\ldots,n_d) \in \Z^d \right\}
$$
is a compact subgroup of $\Omega$ that is contained in $U$. Its index is bounded by $\prod p_j$.

Now, given $f \in C(\Omega,\R)$ and $\omega \in \Omega$, we consider the element $V_\omega$ of $\ell^\infty(\Z^d)$ defined by $V_\omega(n) = f(T^n \omega)$. With the arbitrary choice of $\varepsilon > 0$ above and the resulting $U$ and $\delta > 0$, we consider the following element $V_\omega^p$ of $\ell^\infty(\Z^d)$, $V_\omega^p (n) = f(T^{\tilde n} \omega)$, where $\tilde n = (\tilde n_1 , \ldots , \tilde n_d)$ is defined by $\tilde n_j \in \{0,\ldots, p_j -1 \}$ and $\tilde n_j \equiv n_j \mod p_j$. Thus, $V_\omega^p$ is periodic. We have
\begin{align*}
\|V_\omega - V_\omega^p\|_\infty & = \sup_{n \in \Z^d} |V_\omega(n) - V_\omega^p(n)| \\
& = \sup_{n \in \Z^d} |f(T^n \omega) - f(T^{\tilde n} \omega)| \\
& = \sup_{n \in \Z^d} |f(T^{\tilde n} \omega + (T^n \omega - T^{\tilde n} \omega)) - f(T^{\tilde n} \omega)| \\
& < \varepsilon.
\end{align*}
The first three steps follow by simple rewriting, and the final step follows from the choice of $U$ and the fact that, by construction, $T^n \omega - T^{\tilde n} \omega$ belongs to $U$. This shows that $V_\omega$ is limit-periodic since $\varepsilon > 0$ is arbitrary and $V_\omega^p$ is periodic.

The statement $\mathrm{hull}(V_\omega) = \{ V_{\tilde \omega} : \tilde \omega \in \Omega \}$ follows since both sides are compact and contain $\mathrm{orb}(V_\omega)$ as a dense subset (for the right-hand side, this is a consequence of the minimality of the action). This completes the proof of the proposition.
\end{proof}

Thus, we have seen that a Cantor group that admits a minimal $\Z^d$ action by translations and a continuous sampling function give rise to limit-periodic elements of $\ell^2(\Z^d)$. Let us now turn to the converse. That is, given a limit-periodic element of $\ell^2(\Z^d)$, we want to show that it arises in this way.

\begin{lemma}\label{l.lphulls}
Suppose $V \in \ell^\infty(\Z^d)$ is limit-periodic. Then, $\mathrm{hull}(V)$ is compact and it has a unique topological group structure so that $V$ is the identity element and $\Z^d \to \mathrm{hull}(V)$, $m \mapsto S_m V$ is a homomorphism. Moreover, the group structure is Abelian and there exist arbitrarily small compact open neighborhoods of $V$ in $\mathrm{hull}(V)$ that are finite index subgroups.
\end{lemma}

\begin{proof}
Since $V$ is limit-periodic, we can find for each $\varepsilon > 0$, a periodic $V_p$ with $\|V-V_p\|_\infty < \varepsilon$. Since $\mathrm{orb}(V_p)$ is finite, it follows that $\mathrm{orb}(V)$ is contained in the $\varepsilon$-neighborhood of a finite set. That is, $\mathrm{orb}(V)$ is totally bounded and hence its closure $\mathrm{hull}(V)$ is compact.

Obviously, there is a unique group structure on $\mathrm{orb}(V)$ such that $\Z^d \to \mathrm{orb}(V)$, $m \mapsto S_m V$ is a homomorphism. Our goal is to show that it extends uniquely to a group structure on $\mathrm{hull}(V)$. It suffices to show uniform continuity of the group structure on $\mathrm{orb}(V)$. This will then also show that the resulting extension of the group structure to $\mathrm{hull}(V)$ is Abelian. We have
\begin{align*}
\| S_{m_1 + k_1} V - S_{m_2 + k_2} V \|_\infty &  = \| S_{m_1 - m_2} V - S_{k_2 - k_1} V \|_\infty \\
& \le \| S_{m_1 - m_2} V - V \|_\infty + \| V - S_{k_2 - k_1} V \|_\infty \\
& = \| S_{m_1} V - S_{m_2} V \|_\infty + \| S_{k_1} V - S_{k_2} V \|_\infty .
\end{align*}
Here, the first and the third step follow since translations are isometries and the second step follows from the triangle inequality. Put differently, if $a,b,c,d \in \mathrm{orb}(V)$ and we denote the group operation by $\cdot$, then $\| a \cdot b - c \cdot d\|_\infty \le \|a - c\|_\infty + \| b - d \|_\infty$, which shows the desired uniform continuity.

To prove the last statement about finite index subgroups in small neighborhoods of the identity, let $\varepsilon > 0$ be given. Choose a periodic $V_p \in \ell^\infty(\Z^d)$ with $\| V - V_p \|_\infty < \frac{\varepsilon}{2}$. Also, there are $p_1 , \ldots , p_d \in \Z_+$ such that for all $n = (n_1, \ldots , n_d), k = (k_1 , \ldots , k_d) \in \Z^d$, we have $V_p(n_1 + k_1 p_1 , \ldots, n_d + k_d p_d) = V_p(n_1 , \ldots , n_d)$. In other words, $V_p$ is invariant under $S_m$ for every $m \in (p_1 \Z) \times \cdots \times (p_d \Z)$. Clearly, the closure of $\{ S_m V : m \in (p_1 \Z) \times \cdots \times (p_d \Z) \}$, which we denote by $\mathrm{hull}_p(V)$, is a compact subgroup of $\mathrm{hull}(V)$ of index at most $\prod p_j$. Since $\mathrm{hull}(V)$ is the union of finitely many disjoint translates of $\mathrm{hull}_p(V)$, it follows that $\mathrm{hull}_p(V)$ is also open. By the invariance property of $V_p$, $\mathrm{hull}_p(V)$ is contained in the $\frac{\varepsilon}{2}$-ball around $V_p$, and hence it is contained in the $\varepsilon$-ball around $V$. This completes the proof of the lemma.
\end{proof}

\begin{prop}\label{p.1}
Suppose $V \in \ell^\infty(\Z^d)$ is limit-periodic, but not periodic. Then, there exists a Cantor group $\Omega$ and a minimal $\Z^d$ action by translations, $\{ T^n \}_{n \in \Z^d}$, such that $V(n) = f(T^n \omega_e)$ with the identity element $\omega_e$ of $\Omega$ and a suitable function $f \in C(\Omega,\R)$.
\end{prop}

\begin{proof}
We  may set $\Omega = \mathrm{hull}(V)$, which is a Cantor group by Lemma~\ref{l.lphulls}. The identity element $\omega_e$ is simply $V$ itself. The $\Z^d$ action is given by $T^n = S_n$ with the translations $S_n$ introduced above. Note that this action is indeed an action by translations in the sense of Definition~\ref{d.cgmt}, simply choose $\alpha_j = T^{(0,\cdots, 1, \cdots, 0)}(V)$ with the $j$-th component being $1$.

Let us show that this action is minimal. It suffices to show that for $\omega_1, \omega_2 \in \Omega$ and $\varepsilon > 0$, there is $n \in \Z^d$ such that $\dist (T^n \omega_1 , \omega_2) = \|T^n \omega_1 - \omega_2\|_\infty < \varepsilon$. Since $\Omega = \mathrm{hull}(V)$, we can choose $n_1, n_2 \in \Z^d$ such that $\| \omega_j - T^{n_j} V \|_\infty < \frac{\varepsilon}{2}$, $j = 1,2$. Now set $n := n_2 - n_1$. Putting everything together and using that $T$ is an isometry, we find
\begin{align*}
\|T^n \omega_1 - \omega_2\|_\infty & \le \|T^n \omega_1 - T^{n+n_1} V \|_\infty  + \|T^{n+n_1} V - \omega_2\|_\infty \\
& = \|\omega_1 - T^{n_1} V \|_\infty  + \|T^{n_2} V - \omega_2\|_\infty \\
& < \varepsilon.
\end{align*}

Finally, the continuous function $f$ is simply the evaluation at zero, that is, $f(\omega) = \omega(0)$ (recall that each $\omega$ is a function on $\Z^d$ and may therefore be evaluated at $0 \in \Z^d$). Then, we clearly have the desired formula $V(n) = f(T^n \omega_e)$.
\end{proof}

\begin{definition}
$f \in C(\Omega, \R)$ is called periodic {\rm (}with respect to the $\Z^d$ action $T${\rm )} if for some $p \in (\Z_+)^d$, we have $f(T^p \omega) = f(\omega)$ for every $\omega \in \Omega$.
\end{definition}

\begin{prop}\label{p.denseperiodic}
Suppose $\Omega$ is a Cantor group that admits a minimal $\Z^d$ action by translations, $\{ T^n \}_{n \in \Z^d}$. Then, the set of periodic elements is dense in $C(\Omega, \R)$.
\end{prop}

\begin{proof}
Given $\tilde f \in C(\Omega, \R)$ and $\varepsilon > 0$, we have to find a periodic $f \in C(\Omega, \R)$ with $\|f - \tilde f\|_\infty < \varepsilon$. Since $\tilde f$ is uniformly continuous, there is $\delta > 0$ such that $\dist (\omega_1,\omega_2) < \delta$ implies $|\tilde f(\omega_1) - \tilde f(\omega_2)| < \varepsilon$.

By Proposition~\ref{p.2} and Lemma~\ref{l.lphulls}, we can choose a compact subgroup $\Omega_0$ of $\Omega$ that has finite index and that is contained in the $\delta$-neighborhood of the identity element. Denote the Haar measure on $\Omega_0$ by $\mu_0$ and define
$$
f(\omega) = \int_{\Omega_0} \tilde f(\omega + \omega_0) \, d\mu_0(\omega_0).
$$
By construction, we have $\|f - \tilde f\|_\infty < \varepsilon$. Moreover, $f$ is constant on $\Omega_0$ and also on each of its cosets. This implies that $f$ is periodic.
\end{proof}

We conclude this section with two more observations.

\begin{prop} \label{prop:quot}
Suppose we are given a Cantor group $\Omega$ with a minimal $\Z^d$ action $T$ by translations. Then, for any $f \in C(\Omega,\R)$, $\mathrm{hull}(F(\omega_e))$ is a quotient group of $\Omega$ {\rm (}here, $F(\omega_e) = (f(T^n \omega_e ))_{n \in \Z^d}${\rm )}.
\end{prop}

\begin{proof}
Define $\phi$ by $\phi: \Omega \longrightarrow \mathrm{hull}(F(\omega_e)), \phi(\omega) = F(\omega).$ It is not hard to see that the group structure of $\mathrm{hull}(F(\omega_e))$ is like $F(\omega_1) \cdot F(\omega_2) = F(\omega_1 + \omega_2)$, since $F(T^n \omega_e ) = S_n (F(\omega_e))$. It follows that $\phi$ is a group homomorphism. By Proposition~\ref{p.2}, $\phi$ is surjective. The continuity of $\phi$ follows from the compactness of $\Omega$ and the continuity of $f$. So we have
$$
\mathrm{hull}(F(\omega_e)) \cong \Omega/\mathrm{ker}(\phi),
$$
implying the lemma.
\end{proof}

\begin{prop}\label{prop:iso}
There exists some $f \in C(\Omega,\R)$ such that $\mathrm{hull}(F(\omega_e)) \cong \Omega$.
\end{prop}

\begin{proof}
It suffices to prove that there exists some $f \in C(\Omega,\R)$ such that $\mathrm{ker}(\phi) = \{ \omega_e \}$. Define the function $f: \Omega \to \R$ by $f(\omega) = \dist (\omega_e,\omega)$. Clearly, $f$ is continuous, so there is an associated $F$ (defined as in Proposition~\ref{prop:quot}) $: \Omega \to \ell^{\infty}(\Z^d)$ such that $\mathrm{hull}(F(\omega_e))$ is a quotient group of $\Omega$. Consider $\phi: \Omega \longrightarrow \mathrm{hull}(F(\omega_e)), \phi(\omega) = F(\omega)$. If $F(\omega) = F(\omega_e)$, then $f(\omega) = f(\omega_e)$, that is, $\mathrm{dist}(\omega_e,\omega) = \mathrm{dist}(\omega_e,\omega_e) = 0$, implying $\omega = \omega_e$ and $\mathrm{ker}(\phi) = \{ \omega_e \}$.
\end{proof}

\section{Construction of Distal Limit-Periodic Elements in $\ell^\infty(\Z^d)$} \label{sec:distal}

\begin{definition}
A function $Q(x) : [0,\infty)\to [1,\infty)$ is called an approximation function if both
$$
q(t) = t^{-4} \sup_{x \ge 0} Q(x) e^{-tx}
$$
and
\begin{equation} \label{equ:distalh}
h(t) = \inf_{\kappa_t} \prod^{\infty}_{i = 0} q(t_i)^{2^{-i-1}}
\end{equation}
are finite for every $t > 0$. In \eqref{equ:distalh}, $\kappa_t$ denotes the set of all sequences
$t \ge t_1 \ge t_2 \ge \cdots \ge 0$ with $\sum{t_i} \leq t$.
\end{definition}

\begin{definition}
A sequence $V \in \ell^\infty(\Z^d)$ is called distal if for some approximation function $Q$, we have
$$
\inf_{i \in \Z^d} |V_{i} - V_{i+k}| \ge Q(| k |)^{-1}
$$
for every $ k \in \Z^d \setminus \{0\}$.
\end{definition}

\begin{prop}
If $V \in \ell^{\infty}(\Z^d)$ is distal, then every $\tilde{V} \in \mathrm{hull}(V)$ is also distal.
\end{prop}

\begin{proof}
This follows readily from the definition.
\end{proof}

We will construct a distal sequence in our framework, which is a key step to get our main result.

\begin{lemma} \label{lem:distal}
There exist a Cantor group $\Omega$ that admits a minimal $\Z^d$ action $T$ by translations and an $f \in C(\Omega, \R)$ such that $(f(T^n \omega_e ))_{n \in \Z^d}$ is a distal sequence.
\end{lemma}

\begin{proof}
Choose a collection $\Gamma = \{(n^{(1)}_{v}, n^{(2)}_{v},\cdots, n^{(d)}_{v}): v \in \Z_+ \}$ for which we have $\left (n^{(i)}_{v}\right)^2 \leq n^{(i)}_{v+1} \leq \left (n^{(i)}_v \right)^{2m}$, $\ n^{(i)}_{v} | n^{(i)}_{v+1}$
and $n^{(i)}_{v} \leq C n^{(j)}_{v}$ for some positive integers $m$ and $C$, every $1 \leq i, j \leq d$ and every $v \in \Z_+$.

Define
$$
a^{(i)}_v(p) = q,\quad p = q \mod n^{(i)}_v,
$$
and then define a $d$-dimensional limit-periodic potential as
$$
V(t) = \sum^d_{i = 1} \sum^{\infty}_{v = 1} \frac{a^{(i)}_v(t_i)}{\left(n^{(1)}_{v-1}\right)^2n^{(1)}_v \left(n^{(2)}_{v-1}\right)^2n^{(2)}_v \cdots \left(n^{(d)}_{v-1}\right)^2n^{(d)}_v},
$$
where $t = (t_1, t_2, \cdots, t_d) \in \Z^d$ and $n^{(i)}_0 = 1, 1 \leq i \leq d$. We will show that $V$ is distal.

Write $V_k(t) = \sum^d_{i = 1}  \sum^{k}_{v = 1} \frac{a^{(i)}_v(t_i)}{\left(n^{(1)}_{v-1}\right)^2n^{(1)}_v \left( n^{(2)}_{v-1} \right)^2 n^{(2)}_v \cdots \left(n^{(d)}_{v-1}\right)^2n^{(d)}_v}.$ Clearly, $V_k$ is periodic and $\lim_{k \to \infty} V_k = V$ uniformly. Given $t \neq h \in \Z^d$, we assume that $|t - h| = |t_1 - h_1|$( $|t_i - h_i| \leq |t_1 - h_1|$ for $1 \leq i \leq d$ ). Fix $k$ so that $n^{(1)}_{k-1} \leq |t - h| < n^{(1)}_k.$

If $k \ge 2$, we have
\begin{align*}
|V_k(t) - V_k(h)| \ge \frac{1}{\left(n^{(1)}_{k-1}\right)^2n^{(1)}_k
\left(n^{(2)}_{k-1}\right)^2n^{(2)}_k \cdots \left(n^{(d)}_{k-1}\right)^2 n^{(d)}_k}.
\end{align*}
We also have
\begin{align*}
& |(V(t)-V_k(t)) - (V(h)-V_k(h))| \\
= & \left|\sum^d_{i = 1}  \sum^{\infty}_{v = k+1} \frac{a^{(i)}_v(t_i) - a^{(i)}_v(h_i)}{\left(n^{(1)}_{v-1}\right)^2n^{(1)}_v
\left(n^{(2)}_{v-1}\right)^2n^{(2)}_v \cdots \left(n^{(d)}_{v-1}\right)^2n^{(d)}_v} \right|\\
\leq & dn^{(1)}_k \sum^{\infty}_{v = k+1} \frac{1}{\left(n^{(1)}_{v-1}\right)^2n^{(1)}_v
\left(n^{(2)}_{v-1}\right)^2n^{(2)}_v \cdots \left(n^{(d)}_{v-1}\right)^2n^{(d)}_v} \\
\leq &  \frac{2dn^{(1)}_k}{\left(n^{(1)}_{k}\right)^2n^{(1)}_{k+1}
\left(n^{(2)}_{k}\right)^2n^{(2)}_{k+1} \cdots \left(n^{(d)}_{k}\right)^2 n^{(d)}_{k+1}}\\
\leq & \frac{1}{2 \left(n^{(1)}_{k-1}\right)^2n^{(1)}_k
\left(n^{(2)}_{k-1}\right)^2n^{(2)}_k \cdots \left(n^{(d)}_{k-1}\right)^2 n^{(d)}_k}.
\end{align*}
So we have
\begin{align*}
& |V(t) - V(h)|\\ = & |V_k(t) - V_k(h) + (V(t)-V_k(t)) - (V(h)-V_k(h))|\\
\ge & \frac{1}{\left(n^{(1)}_{k-1}\right)^2n^{(1)}_k
\left(n^{(2)}_{k-1}\right)^2n^{(2)}_k \cdots \left(n^{(d)}_{k-1}\right)^2 n^{(d)}_k}\\
 & - \frac{1}{2 \left(n^{(1)}_{k-1}\right)^2n^{(1)}_k
\left(n^{(2)}_{k-1}\right)^2n^{(2)}_k \cdots \left(n^{(d)}_{k-1}\right)^2 n^{(d)}_k} \\
=& \frac{1}{2 \left(n^{(1)}_{k-1}\right)^2n^{(1)}_k
\left(n^{(2)}_{k-1}\right)^2n^{(2)}_k \cdots \left(n^{(d)}_{k-1}\right)^2 n^{(d)}_k} \\
\ge & \frac{1}{2 \left(n^{(1)}_{k-1}\right)^{2+2m}
\left(n^{(2)}_{k-1}\right)^{2+2m} \cdots \left(n^{(d)}_{k-1}\right)^{2+2m}}\\
\ge & \frac{1}{2 \prod^d_{j = 2} \left( C^{2+2m} \left(n^{(1)}_{k-1}\right)^{2d + 2 dm} \right) }\\
\ge & \frac{1}{2 C^{2d+2dm} |t - h|^{2d + 2 dm}}.
\end{align*}

If $k=1$, we have
\begin{align*}
|V_1(t) - V_1(h)| \ge \frac{1}{\left(n^{(1)}_{0}\right)^2n^{(1)}_1
\left(n^{(2)}_{0}\right)^2n^{(2)}_1 \cdots \left(n^{(d)}_{0}\right)^2 n^{(d)}_1}.
\end{align*}
We also have
\begin{align*}
& |(V(t)-V_1(t)) - (V(h)-V_1(h))| \\
=& \left|\sum^d_{i = 1}  \sum^{\infty}_{v = 2} \frac{a^{(i)}_v(t_i) - a^{(i)}_v(h_i)}{\left(n^{(1)}_{v-1}\right)^2n^{(1)}_v \left(n^{(2)}_{v-1}\right)^2n^{(2)}_v \cdots \left(n^{(d)}_{v-1}\right)^2n^{(d)}_v} \right|\\
\leq & \frac{2 n_1}{\left(n^{(1)}_1\right)^2n^{(1)}_2 \left(n^{(2)}_1\right)^2n^{(2)}_2 \cdots \left(n^{(d)}_1\right)^2n^{(d)}_2}.
\end{align*}
So we have
\begin{align*}
& |V(t) - V(h)|\\ = & |V_k(t) - V_k(h) + (V(t)-V_k(t)) - (V(h)-V_k(h))|\\
\ge & \frac{1}{\left(n^{(1)}_{0}\right)^2n^{(1)}_1
\left(n^{(2)}_{0}\right)^2n^{(2)}_1 \cdots \left(n^{(d)}_{0}\right)^2 n^{(d)}_1} \\ &-  \frac{2 n_1}{\left(n^{(1)}_1\right)^2n^{(1)}_2 \left(n^{(2)}_1\right)^2n^{(2)}_2 \cdots \left(n^{(d)}_1\right)^2n^{(d)}_2} \\
\ge & \frac{1}{2 \left(n^{(1)}_{0}\right)^2n^{(1)}_1 \left(n^{(2)}_{0}\right)^2n^{(2)}_1 \cdots \left(n^{(d)}_{0}\right)^2 n^{(d)}_1}\\
\ge & \frac{1}{2 \prod^d_{j=2}\left( C\left(n^{(1)}_1\right)^d \right ) } \\
\ge & \frac{1}{2 C^{2d+2dm} \left(n^{(1)}_1\right)^{2d + 2 dm}}.
\end{align*}

Thus, we have
$$
|V(t) - V(h)| \ge \begin{cases} \frac{1}{2  C^{2d+2dm} \left(n^{(1)}_1\right)^{2d + 2 dm}}, & 0 < |t-h| < n^{(1)}_1; \\
\frac{1}{2 C^{2d+2dm} |t - h|^{2d + 2 dm}}, & |t -h| \ge n^{(1)}_1.
\end{cases}
$$

Similarly, if $|t - h| = |t_i - h_i|, 1 \leq i \leq d$, we have
$$
|V(t) - V(h)| \ge \begin{cases} \frac{1}{2 C^{2d +2dm} \left(n^{(i)}_1\right)^{2d + 2 dm}}, & 0 < |t-h| < n^{(i)}_1; \\
\frac{1}{2 C^{2d+2dm} |t - h|^{2d + 2 dm}}, & |t -h| \ge n^{(i)}_1.
\end{cases}
$$

Let $M = \max \{n^{(i)}_1 : 1 \leq i \leq d \}$ and
$$
Q(x) = \begin{cases} 2 C^{2d+2dm} M^{2d + 2 dm}, & 0 \leq x < M; \\
2 C^{2d+2dm} x^{2d + 2 dm}, & x \ge M.
\end{cases}
$$
$Q(x)$ is an approximation function (please refer to \cite[Remark 4.5]{dg3}). We have that
$$
|V(t)-V(h)| \ge \frac{1}{Q(|t-h|)}, t \neq h,
$$
which implies that $V$ is distal.

Let $\Omega = \mathrm{hull}(V)$. By Proposition~\ref{p.1}, there exists $f \in C(\Omega, \R)$ such that $f(T^n \omega_e ) = V(n)$, which concludes the proof.
\end{proof}

\section{P\"oschel's Results} \label{sec:poeschel}

In this section we rewrite some of P\"oschel's results from \cite{p}, tailored to our purpose.

Let $\mathfrak{M}$ a Banach algebra of real $d$-dimensional sequences $V = (V_i)_{i \in \Z^d}$ with the operations of pointwise addition and multiplication of sequences. In particular, the constant sequence $1$ is supposed to belong to $\mathfrak{M}$ and have norm one. Moreover, $\mathfrak{M}$ is required to be invariant under translation: if $V \in \mathfrak{M}$, then $\| S_k V \|_{\mathfrak{M}} = \| V \|_{\mathfrak{M}}$ for all $k \in \Z^d$, where as before $S_k V_i = V_{i - k}$.

We denote by $M$ the space of all matrices $A = (V_{i,j})_{i,j \in \Z^d}$ satisfying $A_k = (V_{i, i+k}) \in \mathfrak{M}$, $k \in \Z^d$, that is, $A_k$ is the $k$-th diagonal of $A$ and it is required to belong to $\mathfrak{M}$. In $M$, we define a Banach space
$$
M^s = \{A \in M : \p A\p_s < \infty\},\quad 0 \leq s \leq \infty,
$$
where
$$
\| A\|_s = \sup_{k \in \Z^d} \| A_k\|_{\mathfrak{M}} e^{|k|s}.
$$
Obviously,
$$
M^s \subset M^t,\quad \p \cdot \p_s \ge \p\cdot \p_t,\quad 0 \leq t \leq s \leq \infty.
$$
In particular, $M^{\infty}$ is the space of all diagonal matrices in $M$.

\begin{theorem}[Theorem A, \cite{p}]\label{thm:poeschel1}
Let $D$ be a diagonal matrix whose diagonal $V$ is a distal sequence for $\mathfrak{M}$. Let $0 < s \leq \infty$ and $0 < \sig \leq \min\{1, \frac{s}{2} \}$. If $P \in M^s$ and $\p P \p_s \leq \delta \cdot h(\frac{\sig}{2})^{-1}$, where $\delta > 0$ depends on the dimension $d$ only, then there exists another diagonal matrix $\tilde{D}$ and an invertible matrix $Q$ such that
$$
Q^{-1}(\tilde{D}+P)Q = D .
$$
In fact, $Q,Q^{-1} \in M^{s-\sig}$ and $\tilde{Q} - Q \in M^{\infty}$ with
$$
\p Q-I \p_{s-\sig}, \p V^{-1}-I \p_{s-\sig} \leq C \cdot \p P \p_s,
$$
$$
\p \tilde{D} - D + [P]\p_{\infty} \leq C^2 \cdot \p P \p^2_s,
$$
where $C = \delta^{-1}\cdot h(\frac{\sig}{2})$, and $[\cdot]$ denotes the canonical projection $M^s \to M^{\infty}$. If $P$ is Hermitian, then $Q$ can be chosen to be unitary on $\ell^{2}(\Z^d)$. Note that $h$ is the function \eqref{equ:distalh} associated with $V$.
\end{theorem}

An important consequence of the preceding theorem for discrete Schr\"odinger operators is the following.

\begin{theorem}[Corollary A, \cite{p}]\label{thm:pcor}
Let $V$ be a distal sequence for some translation invariant Banach algebra $\mathfrak{M}$ of $d$-dimensional real sequences. Then for $0 \leq \eps \leq \eps_0$, with $\eps_0 > 0$ sufficiently small, there exists a sequence $\tilde{V}$ with $\tilde{V} -V \in \mathfrak{M}$, $\p \tilde{V} - V \p_{\mathfrak{M}} \leq \frac{\eps^2}{\eps^2_0}$, such that the discrete Schr\"odinger operator
$$
(\tilde{H}u)_i = \eps \sum_{|l|_1 = 1}u_{i + l} + \tilde{V}_i u_i, \quad i \in \Z^d
$$
has eigenvalues $\{V_i : i \in \Z^d \}$ and a complete set of corresponding exponentially localized eigenvectors with decay rate
$1 + \log\frac{\eps_0}{\eps}$.
\end{theorem}

\section{Proof of Theorem~\ref{thm:main}} \label{sec:localization}

We are now proving our main result, Theorem~\ref{thm:main}. By Lemma~\ref{lem:distal}, there exist a Cantor group $\Omega$ that admits a minimal $\Z^d$ action $T$ by translations and an $f \in C(\Omega, \R)$ such that $(f(T^n \omega_e ))_{n \in \Z^d}$ is a distal sequence. Clearly, $C(\Omega,\R)$ will induce a class of limit-periodic potentials. We denote it by $\mathcal{C}$, and one can check that this class is a translation invariant Banach algebra with the $\ell^\infty$-norm. By Theorem~\ref{thm:pcor}, there exists a sufficiently small $\eps_0 > 0$ such that for $0 < \eps \leq \eps_0$, there is a sequence $\tilde{V} \in \mathcal{C}$ with $\p\tilde{V}-V \p_{\infty} \leq \frac{\eps^2_0}{\eps^2}$ so that the discrete Schr\"odinger operator
$$
(Hu)_n = \sum_{|m-n|_1 = 1}u_{m} + \frac{\tilde{V}_n}{\eps}u_n,\quad n \in \Z^d
$$
has eigenvalues $\{\frac{V_n}{\eps}, n \in \Z^d\}$ and a total set of corresponding exponentially localized eigenvectors with decay rate $r = 1+ \log{\frac{\eps_0}{\eps}}$. There exists a sampling function $\tilde{f} \in C(\Omega, \R)$ such that $\tilde{f}(T^n \omega_e ) = \frac{\tilde{V}_n}{\eps}$ since $\tilde{V} \in \mathcal{C}.$

More explicitly, for the Schr\"odinger operator $H$ associated with the potential $\tilde{f}(T^{n} \omega_e )$, denote its matrix representation with respect to the standard orthonormal basis of $\ell^2(\Z^d)$, $\{ \delta_n \}_{n \in \Z^d}$, by the same symbol. P\"oschel's theorem also implies that there exists a unitary $Q: \ell^2(\Z^d) \to \ell^2(\Z^d)$ (with corresponding matrix denoted by the same symbol) such that
\begin{equation}\label{equ:jacob1}
H \cdot Q = Q \cdot D,
\end{equation}
where $D = (D_{i,j})_{i,j \in \Z^d}$ is a diagonal matrix with the diagonal $(\frac{V_n}{\eps})_{n \in \Z^d}$.

Write $Q_n = (Q_{i, i+n})_{i \in \Z^d}$ as the $n$-th diagonal of $Q$ (note that $n \in \Z^d$). By Theorem~\ref{thm:poeschel1} we have that $Q \in M^{r}$, where $r > 0$ and $M^r$ is a space of matrices associated with the Banach algebra $\mathcal{C}$. (Note that $Q \in M^{r}$ follows from \cite[Proof of Corollary~A]{p}.) Since $Q \in M^r$, we have $\p Q\p_r = \sup_{n \in \Z^d} \p Q_n \p_{\infty} e^{|n|r} < C$ where $C$
is a constant. So $\p Q_n \p_\infty < C e^{-r|n|}, \forall n \in \Z^d$. Let $Q^{(n)} = (Q_{i, n})_{i \in \Z^d}$ be the $n$-th column of $Q$, that is, $Q^{(n)}$ is an eigenfunction of $H$ with respect to the eigenvalue $\frac{V_n}{\eps}$. Since $Q_{k,n} = Q^{(n)}(k) = Q_{k, k +(n-k)}$, $Q^{(n)}(k)$ is also an entry in $Q_{n-k}$, and so $|Q^{(n)}(k)| < Ce^{-r|n-k|}$. $C$ and $r$ are independent of $n$, so the corresponding Schr\"odinger operator $H$ has ULE. This property is strong enough to imply that the pure point spectrum of $H$ is independent of $\om$ \cite{j}, that is, it is phase stable. In order to see this more explicitly, we will prove it in our framework, and furthermore show that for other $\om$'s, the associated Schr\"odinger operator still has ULE with the same constants $C$ and $r$. Note that the latter property does not follow from Jitomirskaya's result in \cite{j}.

\begin{lemma}
Suppose we are given matrices $A,B \in \R^{\Z^d \times \Z^d}$, one of which has only finitely many non-zero diagonals. Then,
we have for the $k$-th diagonal of $Z = AB$,
$$
Z_k = \sum_{l \in \Z^d} A_l \cdot T^{l}(B_{k-l}),
$$
where $\cdot$ is the pointwise multiplication {\rm (}i.e., $A_l \cdot T^{l}(B_{k-l})$ is still a sequence{\rm )} and $T$ is the translation defined by $T^{l}(B_{k-l})(n) = B_{k-l}(n+l)$ for $n \in \Z^d$.
\end{lemma}

\begin{proof}
Since for $n,k \in \Z^d$, we have
\begin{align*}
z_{n,n+k} = & \sum_{t \in \Z^d} a_{n,t}b_{t,n+k} \\
= & \sum_{l \in \Z^d} a_{n,n+l}b_{n+l,n+k} \\
= & \sum_{l \in \Z^d}a_{n,n+l}b_{n+l,n+l+k-l},
\end{align*}
the lemma follows.
\end{proof}

Given an $\om \in \Omega$, we have that $(\tilde{f}(T^{n} \om ))_{n \in \Z^d} \in \mathrm{hull}((\tilde{f}(T^{n} \omega_e ))_{n \in \Z^d})$. If $\omega$ is in the orbit of $\omega_e$, that is, $\om = T^{t} \omega_e$ for some $t \in \Z^d$, ULE with the same constants and eigenvalues follows from unitary operator equivalence directly. However, we write this out in detail so that we see clearly what happens in the case where $\om$ can only be approximated by elements of the form $T^{t} \omega_e$.

By the previous lemma, \eqref{equ:jacob1} is equivalent to the following form:
$$
\forall k \in \Z^d : \qquad \sum_{l \in \Z^d} H_l \cdot T^l (Q_{k-l}) = \sum_{l \in \Z^d} Q_l \cdot T^l(D_{k-l}).
$$
If the potential is replaced by $\tilde{f}(T^{n+t} \omega_e)$, the corresponding matrix $\tilde{H}$ has $\tilde{H}_j(n)=H_j(n+t) = T^t(H_j)(n)$. We have
$$
\forall k \in \Z^d : \qquad \sum_{l \in \Z^d} T^t(H_l) \cdot T^{t+l} (Q_{k-l}) = \sum_{l \in \Z^d} T^t(Q_l) \cdot T^{l+t}(D_{k-l}).
$$
We let $\tilde{Q}_k(n) = Q_k(n+t),\ k \in \Z^d $ and $\tilde{D}_0(n) = D_0(n+t)$. Reversing the steps above, this means that
$$
\tilde H \cdot \tilde Q = \tilde Q \cdot \tilde D.
$$
We can conclude that $\tilde{H}$ has the pure point spectrum $\overline{\{\frac{V_{n+t}}{\eps} : n \in \Z^d \}} = \overline{\{\frac{V_n}{\eps}:n \in \Z^d \}}$. Moreover, $\tilde{Q}$ is the eigenfunction matrix of $\tilde{H}$, and for any $n,k \in \Z$, $|\tilde{Q}_k (n)| = |Q_k (n+t)| \leq Ce^{-r|k|}$. So for the eigenfunction $\tilde{Q}^{(k)}$ of $\tilde{H}$, we still have $|\tilde{Q}^{(k)}(n)| < Ce^{-r|k-n|}$, and hence ULE with the same constants follows.

If $\lim_{m \to \infty} T^{t_m} \omega_e = \om$, that is, $\tilde{f}(T^{n} \om ) = \lim_{m \to \infty}\tilde{f}(T^{n+t_m} \omega_e)$, then for $\tilde{f}(T^{n+t_m} \omega_e )$, we have already seen that
\begin{equation}\label{equ:shift}
\tilde{H}^{\langle m \rangle}\cdot \tilde{Q}^{\langle m \rangle} = \tilde{Q}^{\langle m \rangle} \cdot \tilde{D}^{\langle m \rangle},
\end{equation}
where the notation $\langle m \rangle$ means that the matrices are corresponding to $\tilde{f}(T^{n+t_m} \omega_e )$. Let $\tilde{Q}^{\langle m \rangle}_k$ be the $k$-th diagonal of $\tilde{Q}^{\langle m \rangle}$. Clearly, $\tilde{Q}^{\langle m \rangle}_k(n) = Q_k(n+t_m)$. There exists some $\tilde{f}_k \in C(\Omega, \R)$ such that $\tilde{Q}^{\langle m \rangle}_k(n) = Q_{k}(n+t_m) = \tilde{f}_k(T^{n + t_m} \omega_e)$ since $Q_k \in \mathcal{C}$. So $\lim_{m \to \infty} \tilde{Q}^{(m)}_k(n)  = \lim_{m \to \infty} \tilde{f}_k (T^{n+t_m}\omega_e) = \tilde{f}_k (T^{n}\om)$, and we denote $\tilde{f}_k (T^{n} \om)$ by $\tilde{Q}^{\langle \infty \rangle}_k (n)$. Similarly, $\lim_{m \to \infty} \tilde{D}^{\langle m \rangle}$ exists and $\tilde{D}^{\langle \infty \rangle}_0(n) = f(T^n \om )$, where $\tilde{D}^{\langle \infty \rangle}_0$ is the 0-th diagonal of $\tilde{D}^{\langle \infty \rangle}$. Thus, as we let $m \to \infty$, \eqref{equ:shift} takes the following form:
\begin{equation}\label{equ:limit}
\tilde{H}^{\langle \infty \rangle} \cdot \tilde{Q}^{\langle \infty \rangle} = \tilde{Q}^{\langle \infty \rangle} \cdot \tilde{D}^{\langle \infty\rangle},
\end{equation}
where $\tilde{H}^{\langle \infty \rangle}$ is (the matrix representation of) the Schr\"odinger operator with the potential $\tilde{f}(T^n \om)$. Equation~\eqref{equ:limit} implies that $\tilde{H}^{\langle \infty \rangle}$ has the pure point spectrum  $\overline{\{\frac{V_n}{\eps} : n \in \Z^d \}}$, and its eigenfunctions are uniformly localized since $|(\tilde{Q}^{\langle \infty \rangle})^{(k)}(n)| < Ce^{-r|k-n|}$ for any $k, n \in \Z^d$, where $(\tilde{Q}^{\langle \infty \rangle})^{(k)}$ is the $k$-th column of $\tilde{Q}^{\langle \infty\rangle}$. This finishes the proof.


\begin{thebibliography}{10}

\bibitem{a} A.\ Avila, On the spectrum and Lyapunov exponent of limit periodic Schr\"odinger operators,
\textit{Commun.\ Math.\ Phys.}\ \textbf{288} (2009), 907--918


\bibitem{dg1} D.\ Damanik, Z.\ Gan, Spectral properties of limit-periodic Schr\"odinger operators,
\textit{Commun.\ Pure Appl.\ Anal.}\ \textbf{10} (2011), 859--871

\bibitem{dg2} D.\ Damanik, Z.\ Gan, Limit-periodic Schr\"{o}dinger operators in the regime of positive Lyapunov exponents,
\textit{J.\ Funct.\ Anal.}\ \textbf{258} (2010), 4010--4025

\bibitem{dg3} D.\ Damanik, Z.\ Gan, Limit-periodic Schr\"{o}inger operators with uniformly localized eigenfunctions,
\textit{J. d'Analyse Math},  \textbf{115} (2011), 33--49

\bibitem{dj2} R.\ del Rio, S.\ Jitomirskaya, Y.\ Last, B.\ Simon, What is localization?, \textit{Phys.\ Rev.\ Lett.}\ \textbf{75} (1995), 117--119

\bibitem{dj} R.\ del Rio, S.\ Jitomirskaya, Y.\ Last, B.\ Simon, Operators with singular continuous spectrum,
IV.~Hausdorff dimensions, rank one perturbations, and localization, \textit{J.\ Anal.\ Math.}\ \textbf{145} (1997), 312--322

\bibitem{g} Z.\ Gan, An exposition of the connection between limit-periodic potentials and profinite groups, \textit{Math.\ Model.\ Nat.\ Phenom.}\ \textbf{5:4} (2010), 158--174

\bibitem{j}  S.\ Jitomirskaya, Continuous spectrum and uniform localization for ergodic Schr\"odinger operators,
\textit{J.\ Funct.\ Anal.}\ \textbf{145} (1997), 312--322


\bibitem{p} J.\ P\"oschel, Examples of discrete Schr\"odinger operators with pure point spectrum,
\textit{Commun.\ Math.\ Phys.}\ \textbf{88} (1983), 447--463



\end{thebibliography}
\end{document}